\documentclass[12pt]{amsart}
\usepackage{graphicx,comment,color}
\usepackage{amssymb}
\usepackage{amsthm}
\usepackage{amsmath}
\usepackage[all]{xy}

\newtheorem{thm}{Theorem}[section]
\newtheorem{cor}[thm]{Corollary}
\newtheorem{lem}[thm]{Lemma}
\newtheorem{prop}[thm]{Proposition}
\theoremstyle{definition}
\newtheorem{defn}[thm]{Definition}
\theoremstyle{remark}
\newtheorem{rem}[thm]{Remark}
\numberwithin{equation}{section}

\newcommand{\UU}{\mathcal{U}}
\newcommand{\BB}{\mathcal{B}}\usepackage{color}
\usepackage{blkarray}
\begin{document}

\title[$\;$]{Extending Harvey's Surface Kernel Maps}%
\author{Jane Gilman}%
\address{Mathematics \& CS department, Rutgers University, Newark, NJ 07102}%
\email{gilmanjp@gmail.com; gilman@rutgers.edu}%
\dedicatory{In Memory of Bill Harvey}

\thanks{Some of this work was carried out while the author was a supported Visiting Fellow at Princeton University or was supported by the Rutgers Research Council}

\date{April 2, 2021  and March 22, 2021 revisions of July 17, 2020 version}

\begin{abstract}
Let $S$ be a compact Riemann surface and $G$ a group of conformal automorphisms of $S$ with $S_0=S/G$.  $S$ is a  finite regular branched cover of $S_0$. If $U$ denotes the unit disc, let  $\Gamma$ and $\Gamma_0$ be the Fuchsian groups with  $S= U/\Gamma$ and $S_0 = U/\Gamma_0$. There is a group homomorphism  of $\Gamma_0$ onto $G$ with kernel $\Gamma$ and this is termed  a {\sl surface kernel map}. Two surface kernel maps are equivalent if they differ by an automorphism of $\Gamma_0$. In his $1971$ paper Harvey showed that when $G$ is a cyclic group, there is a unique simplest representative for this equivalence class. His result  has played an important role in establishing subsequent results about conformal automorphism groups of surfaces. We extend his result to some surface kernel  maps onto arbitrary finite groups. These can be used along with the Schreier-Reidemeister Theory to find a  set of generators for $\Gamma$ and  the action of $G$ as an outer automorphism group on the fundamental group of $S$ putting  the action on  the fundamental group and the induced action on homology into a relatively simple format. As an example we compute generators for the fundamental group and a homology  basis together with the action of $G$  when $G$ is $\mathcal{S}_3$, the symmetric group on three letters. The action of $G$ shows that the homology basis found is not an adapted homology basis. \end{abstract}

\maketitle

\section{Introduction}

Let $S$ be a compact Riemann surface and $G$ a finite group of conformal automorphisms of $S$. Let $S_0$ be the quotient, $S_0=S/G$. We assume that $g$ is the genus of $S$ and $g_0$ that of $S_0$. The Riemann-Hurwitz relation (equation \ref{eq:RH})
gives the connection between $g$ and $g_0$ which depends upon the branch structure of the covering $S \rightarrow S_0$.   If $U$ denotes the unit disc, let  $\Gamma$ and $\Gamma_0$ be Fuchsian groups so that $S= U/\Gamma$ and $S_0 = U/\Gamma_0$. There is a group homomorphism  $\phi$ of $\Gamma_0$ onto $G$ with kernel $\Gamma$ and this is termed  a {\sl surface kernel map}. Two surface kernel maps are equivalent if they differ by an automorphism of $\Gamma_0$.
In 1966 W. J. Harvey began his study of conformal automorphism groups \cite{Harvey1} by looking at cyclic groups.
      In his $1971$ paper \cite{Harvey2} he showed that when $G$ is a cyclic group, there is a unique simplest representative in a surface kernel equivalence class.       This result  has played an important role in establishing subsequent results about conformal automorphism groups of surfaces. It has been used widely and has been cited in over $67$ papers, too many to cite here. Here we  extend his results to some surface kernels of maps onto arbitrary finite groups.

Riemann surfaces is an old topic and can be thought of as the cradle of much of modern mathematics including complex analysis, group theory, differential and algebraic geometry, algebraic topology,  combinatorial group theory and computation.  Since Riemann surfaces  have been studied in a number of different settings, many settings use  their own language. Many terms are equivalent.  Thus use of the {\sl surface kernel map} is equivalent to use of the {\sl generating function}. The {\sl mapping-class group} is the same as the {\sl Teichm\"uller modular group} and both act on Teichm\"uller space. This action in earlier times was more often termed  the action of the {\sl Teichm\"uller modular group} or simply {\sl the modular group}.

The  problems we consider here can be approached in a number of ways: in terms of curve lifting and defining subgroups of coverings or using the Reidemeister-Schreier theory. Our main approach here is the Reidemeister-Schreier theory but we also use curve lifting when it simplifies the argument.

While much of this paper is expository, especially an exposition of past terminology and alternative terminology, it contains a few new results, namely an extension of Harvey's list to some nonabelian groups (Theorem
 \ref{thm:ext})
and a computation of the action of $\mathcal{S}_3$ on the fundamental group (Section \ref{sec:S_3}) and on a non-adapted  integral homology basis where the action of elements of the group are easily computed.

This volume represents the culmination of a series of AMS special sessions on automorphisms of Riemann surfaces. Results on adapted homology bases, fundamental groups and Schreier-Reidemeister Theory  (S-R Theory) were presented in those sessions and published elsewhere. (See  the bibliography and papers cited in those papers). As a final paper for these sessions, we review the S-R theory and we give an example of its application when $G= {\mathcal{S}}_3$.

The organization of this paper is as follows:

Section \ref{sec:history} contains some history and methods. Further historical remarks are interspersed throughout the paper. Section \ref{sec:Julynotation} contains further  notation.  Harvey's results and extensions to arbitrary finite groups appear in section \ref{sec:HarveyThm} with Theorems \ref{thm:ext} and \ref{thm:furtherext}.  Coset representatives are discussed in section \ref{sec:cosetreps} and section \ref{sec:SRSum} gives the Schreier-Reidemeister theorem followed by section \ref{sec:SR} which gives its application to our situation. Finally, Section \ref{sec:S_3} gives the full details of the computation for the action of the symmetric group on three letters. We end in Section \ref{sec:questions} with some open questions.

\section{History and Methods} \label{sec:history}
The term surface kernel first appears in papers of Maclachlan \cite{MacL} and Harvey \cite{Harvey2}, but it is likely that it arose in Macbeath's 1966 Dundee lecture notes \cite{MacB} or discussions following the lectures and also papers \cite{MacB2, MacB3}. In later years Maclachlan \cite{mcLM} used the term generating function. Harvey used automorphisms of $\Gamma_0$ to obtain surface kernel maps for a cyclic group in a nice form,  from which he could derive results about the action of cyclic groups of conformal automorphisms.

Here we extend his list to include maps that  work for all finite groups with  $n = o(G)$. Additional background can be found in \cite{Matrix}. The concept of an automorphism with an adapted basis was used for a prime order group and later extended to an arbitrary group.
The initial idea for an adapted basis appeared in \cite{thesis} and subsequent extensions in \cite{Jalg, Link, Adapted} and papers referenced there.
Here we actually find the action on a homology basis 
for the symmetric group of order $6$, but it fails to be an adapted basis.

Our method is to choose minimal Schreier right coset representatives for the map $\phi: \Gamma_0 \rightarrow  G$  and combine the choice with Harvey's original surface kernel maps for cyclic groups.

\section{Notation}\label{sec:Julynotation}
In this section we fix notation and state some background results. We assume that the reader is familiar with the covering theory of surfaces and curve lifting. A good reference is \cite{Springer}. This will also be relevant in section \ref{sec:kernelLIFT}.

We let $n$ be the order of $G$ with $\pi: S \rightarrow S/G=S_0$ the projection. If $p$ is fixed by an element of $G$, then $G_p$, the stabilizer of $p$ is cyclic, say of order $m_p$.
We let $p_0= \pi(p)$. Then $\pi^{-1}(p_0)$ consists of $n/m_p$ distinct points each with a cyclic of stabilizer of order $m_p$.  The point $p_0 \in S_0$ is called a {\sl branch point} and $p$ is called a {\sl ramification point}. If the covering has $r$ branch points of orders $m_1,...,m_r$, then over $p_0$ instead of $n$ points,
there are ${\frac{n}{m_p}}$ points.


Since $S$ is compact, $\Gamma$ is isomorphic to the fundamental group of $S$.  By abuse of language we refer to elements of $\Gamma$ as both curves and group elements.

  While $\Gamma_0$  is not isomorphic to the fundamental group of $S_0$, it is the homomorphic image of
  $\pi_1(S/G - \{{\mbox{branch points}} \})$.
    Thus using the composed map: $$\pi_1(S/G-\{{\mbox{branch  points}} \}) \rightarrow \Gamma_0 \xrightarrow{\phi} G$$ one can talk about the images in $G$ of curves in $\Gamma_0$.

 Knowing the isomorphism between the fundamental group of $S$ and $\Gamma$, we speak about the action of $G$ on $\pi_1(S)$ which is induced by the action of $\Gamma_0/\Gamma$ on $\Gamma$, given by conjugation and the inverse of the quotient isomorphism ${\overline{\phi}}: \Gamma_0/\Gamma \rightarrow G$. The action is well-defined up to inner automorphisms of $\pi_1(S)$ and $\Gamma$.

Since $G$ is a finite group, $S_0$ has no cusps and thus the presentation for $\Gamma_0$ contains no parabolic elements. We may assume that  $\Gamma_0$ has presentation

\begin{equation} \label{eq:presentation}
\Gamma_0=
\langle a_1,...,a_{g_0}, b_1, ..., b_{g_0}, x_1,...,x_r  \; |\; (\Pi_{i=1}^{g_0} [a_i,b_i])x_1 \cdots x_r =1; \;\; x_i^{m_i} =1 \rangle
\end{equation}

Here the $m_i$ are positive integers each  at least $2$ and $[v,w]$ denotes the multiplicative commutator of elements $v$ and $w$. Any elliptic element of the group $\Gamma_0$ is conjugate to one of the $x_i$. No two different $x_i$'s are conjugate.

Knowing the surface kernel map is equivalent to knowing {\sl the generating vector}, the vector of length  $2g_0 +r$
given by

\begin{equation} (\phi(a_1), ... ,  \phi(a_{g_0}),\phi(b_1),..., \phi(b
_{g_0}), \phi(x_1), ..., \phi(x_r)).
\end{equation}

The presentation \ref{eq:presentation} is equivalent to there being $r$ branch points of order $m_i, i=1,...,r$, and thus by the Riemann-Hurwitz relation we have

\begin{equation} \label{eq:RH}
2g-2 = n(2g_0 -2) + n \cdot \Sigma_{i=1}^r (1- {\frac{1}{m_i}})
\end{equation}

We note for future use that if one has eliminated generators and relations and one has exactly one relation in which every generator and its inverse occurs exactly once, then it is equivalent by  the standard algorithm in \cite{Springer} to a standard surface presentation.  The algorithm in \cite{Springer} is given geometrically in terms of cutting and pasting sides of a fundamental polygon but translates to a purely algebraic algorithm as a set of generators for the group with a single defining relation and can be viewed as giving a fundamental polygon. A standard surface presentation is one given by $2g$ generators and one relation which is the product of $2g$ commutators.

By abuse of notation we use $=$ to denote that words or curves  in  a group are equal, freely equal or
 are homologous, that is equal  as elements of an integral homology basis. We write the integral homology additively. For simplicity we use the same notation for a curve, its equivalence class as an element of the fundamental group and its image in homology.

\section{Harvey's results and Extensions} \label{sec:HarveyThm}
Harvey used a list of automorphisms of $\Gamma_0$ which preserved the equivalence class of a cyclic surface kernel which he used to give the simplest representation for a surface kernel map in the equivalence class.
We extend his results here to arbitrary finite groups $G$ as this can be useful in finding adapted integral homology bases and generating sets.
We let $A_i=\phi(a_i); B_i = \phi(b_i), i= 1,...,g_0$ and $\xi_j = \phi(x_j), j = 1,...,r$.

Harvey lists automorphisms of $\Gamma_0$ by their effect on the $2g_0+r$ generators and then gives their corresponding effect on the images of the curves in $G$ under the representation $\phi: \Gamma_0 \rightarrow G$.


\subsection{We recall Harvey's list and actions.}


$\;\;\;$

$\;\;\;\;$

1. $\UU_1$: 

$\;\;$ $a_1 \rightarrow a_1b_1$

$\;\;$ $v \rightarrow v \;\; \forall \mbox{ generators }  v \ne a_1$

{\it Action on representations:}

$\;\;$ $A_1 \rightarrow A_1 +B_1$

$\;\;$ All other  images remain fixed.

2. $\UU_2$:

$\;\;$ $a_1 \rightarrow a_1b_1$

$\;\;$ $b_1 \rightarrow a_1^{-1}$

$\;\;$ $v \rightarrow v \; \forall \mbox{ generators }  v \ne a_1 \mbox{ or } b_1$

{\it Action on representations:}

$\;\;$ $A_1 \rightarrow A_1 + B_1$

$\;\;$ $B_1 \rightarrow  - A_1 $

 $\;\;$ All other images remain fixed.

3.  $\UU_3$\footnote{There is a typo on in the third line of the table on page 397 of \cite{Harvey2} }:

$\;\;$ $ x_j \rightarrow a_2x_ja_2^{-1}$,  $j = 1,...,r$

 $\;\;$ $ a_1 \rightarrow a_2a_1$

 $\;\;$ $ a_2 \rightarrow b_1a_2b_1^{-1}$

$\;\;$ $ b_2 \rightarrow a_2b_2a_2^{-1}b_1^{-1}$

 $\;\;$$ v \rightarrow v \; \forall \mbox{ generators }  v \ne a_1, a_2  \mbox{ or } b_2$

$\;\;$ All other images remain fixed.

{\it Action on representations:}

 $\;\;$ $A_1 \rightarrow A_1 + A_2$

$\;\;$ $B_2 \rightarrow B_2-B_1$

4. $\UU_4$:

 $\; \;$ $ x_r \rightarrow a_1x_ra_1^{-1}$

 $\;\;$ $a_1 \rightarrow [a_1, x_r^{-1}]a_1$

 $\;\;$ $b_1 \rightarrow b_1a_1^{-1}x_ra_1$

 $\;\;$ $v \rightarrow v  \; \forall \mbox{ generators }  v \ne a_1 \mbox{ or } b_1$

{\it Action on representations:}
 $\;\;$ $B_1 \rightarrow B_1 + \xi_r.$

All other images remain unchanged.

Since the group $G$ is abelian, the image of $x_r$ is not changed.
\subsection{Extensions}

We list the actions on representations of automorphisms of $\Gamma_0$ listed above  which work for any group $G$.

1. $V_1$:

$\;\;$ $a_1 \rightarrow a_1b_1$

$\;\;$ $v \rightarrow v \;\; \forall \mbox{ generators }  v \ne a_1$

{\it Action on representations:}

$\;\;$ $A_1 \rightarrow A_1B_1$

$\;\;$ All other  images remain fixed.

2. $V_2$:

$\;\;$ $a_1 \rightarrow a_1b_1$

$\;\;$ $b_1 \rightarrow a_1^{-1}$

$\;\;$ $v \rightarrow v \; \forall \mbox{ generators }  v \ne a_1 \mbox{ or } b_1$

 {\it Action on representations:}

$\;\;$ $A_1 \rightarrow A_1B_1$

$\;\;$ $B_1 \rightarrow A_1^{-1} $

 $\;\;$ All other images remain fixed.

 3. $V_3$:

$\;\;$ $ x_j \rightarrow a_2x_ja_2^{-1}$, $j = 1, ...
, r$

$\;\;$ $ a_1 \rightarrow a_2a_1$

$\;\;$ $ a_2 \rightarrow b_1a_2b_1^{-1}$

$\;\;$ $ b_2 \rightarrow a_2b_2a_2^{-1}b_1^{-1}$

$\;\;$ $ v \rightarrow v \; \forall \mbox{ generators }  v \ne a_1 \mbox{ or } b_1$

$\;\;$ All other images remain fixed.

{\it Action on representations:}

$\;\;$ $A_1 \rightarrow A_1A_2$

$\;\;$ $A_2 \rightarrow B_1A_2B_1^{-1}$

$\;\;$ $B_2 \rightarrow A_2B_2A_2^{-1}B_1^{-1}$

4. $V_4$:

 $\; \;$ $ x_r \rightarrow a_1x_ra_1^{-1}$

 $\;\;$ $a_1 \rightarrow [a_1,x_r^{-1}]a_1$

 $\;\;$ $b_1 \rightarrow b_1a_1^{-1}x_ra_1$

 $\;\;$ $v \rightarrow v  \; \forall \mbox{ generators }  v \ne a_1 \mbox{ or } b_1$

{\it Action on representations:} (assuming $\phi(a_1) \in \langle  \phi(x_r)\rangle$)

 $\;\;$ $B_1 \rightarrow B_1\xi_r$

 $\;\;$ All other images remain fixed.

\begin{thm} \label{thm:ext} Harvey's ${\UU}_1, \;  {\UU}_2  \; \mbox{and } \; {\UU}_3$ remain the same no matter the nature of $G$ (cyclic or abelian or not) but we write them as $V_1,V_2,V_3$ respectively.
 If $\phi(a_1) \in \langle  \phi(x_r)\rangle $   then ${\UU}_4$ will work. We write this as $V_4$ to indicate it has a modified application.
 If $G$ is abelian, then Harvey's $\UU_4$ or $V_4$  work without the requirement that $\phi(a_1) \in \langle  \phi(x_r)\rangle $

\end{thm}

\begin{proof}
Verify that the long defining relation still holds and that the elliptic relations are unchanged.
\end{proof}

\subsection{A Further Automorphism}

\vskip .1in

 Harvey's map $\BB_j$ is given by

\vskip .1in

$x_i \rightarrow x_i$, $i=1,...,r$

$(a_i, b_i) \rightarrow (a_i, b_i)$ if $i \ne j,j+1$

$(a_j,b_j) \rightarrow (a_{j+1},b_{j+1})$

$(a_{j+ 1}, b_{j+ 1}) \rightarrow  (c_{j+1}a_jc_{j+1}^{-1}, c_{j+1}b_jc_{j+1}^{-1})$

where $c_{j+1} = [a_{j+1},b_{j+1}].$

\vskip .1in

{\it  And $\BB_j$ induces the mapping}

 $\xi_i \rightarrow \xi_i, \;  i= 1,...,r$

 $(A_i,B_i) \rightarrow (A_i,B_i),\; i \ne j,\; j+1$

 $(A_j,B_j) \rightarrow (A_{j+1},B_{j+1}),  i \ne j,j+1$

 $(A_{j+1},B_{j+1}) \rightarrow (A_j,B_j)$

 for $G$ a cyclic group where $c_{j+1}= [a_{j+1},b_{j+1}]$.

\vskip .2in

We call the new map $\hat{\BB_j}$.

We let  $C_{j+1}= [A_{j+1},B_{j+1}]$ and define the action of $\hat{\BB_j}$ by

 $\xi_i \rightarrow \xi_i, i= 1,...,r$

 $(A_i,B_i) \rightarrow (A_i,B_i), i \ne j,j+1$

 $(A_j,B_j) \rightarrow (A_{j+1},B_{j+1}),  i \ne j,j+1$

 $(A_{j+1},B_{j+1}) \rightarrow (C_{j+1}A_jC_{j+1}^{-1},C_{j+1}B_jC_{j+1}^{-1}).$

\begin{thm} \label{thm:furtherext}

          If $\phi(a_{j+1}) \in \langle \phi(b_{j+1}) \rangle$, then $\hat{\mathcal{B}}_j$ allows us to interchange the pairs $(A_i,B_i)$ and $(A_{j+1},B_{j+1}) $ if $i \ne j, j+1$ leaving the elliptic images fixed.


\end{thm}

\section{Background: Summary of the Reidemeister-Schreier theory} \label{sec:SRSum}

  In this section we first recall standard terminology, facts and theorems for the Reidemeister-Schreier Theory (see section 2.3 of \cite{MKS}.)
We choose a set of minimal Schreier right coset representatives for $\Gamma_0$ modulo $\Gamma$ as defined below and denote an arbitrary representative by $K$. We let $\tau$ be the rewriting system so that the generators of $\Gamma$ are given by
 $S_{K,a} = Ka{\overline{Ka}}^{-1}$ where $\overline{X}$ denotes the coset representative of $X$ and $a$ varies over a set of generators of $\Gamma_0.$  Our first result comes from the application of the Reidemeister-Schreier theorem to our situation (Theorem 2.9 page 94 \cite{MKS}). It would be nice to simplify the notation for this,  but I have yet to find a better notation.

\vskip .1in

First we summarize material from \cite{MKS}.

\subsubsection{Schreier Representatives} \label{sec:cosetreps}

Let $G$ be a finite group of order $n$,  $\Gamma_0$ a group with known presentation, $\phi$ a group  homomorphism of $\Gamma_0$ onto  $G$ and $\Gamma$ the kernel of $\phi$.  We let $K_1, ..., K_n$ be a set of right coset representatives for $\Gamma_0$ with $K$ denoting an arbitrary one of these chosen right-coset representatives.

Recall that the length of a coset is the length of the shortest word in the coset.

\begin{defn}

A  {\sl Schreier right coset function} is a right coset function  where the initial segment of any coset representative is again a right coset representative. It is a minimal if the length of any coset does not exceed the length of any coset it represents. 
 We call the set of cosets a {\sl Schreier} system and denote the coset of word $W$ in $\Gamma_0$ by ${\overline{W}}$.
\end{defn}

Note that we can always choose a set of {\sl minimal} Schreier representatives.
A rewriting process is  process that takes a word that is in $\Gamma$ but that is given in the generators of $\Gamma_0$ and writes it as a word in generators for $\Gamma$. A {\sl Reidemeister-Schreier rewriting process} is one that uses a Schreier system. Here we denote our rewriting process by $\tau$.

There are other ways to choose coset representatives. For example one may order the generators and their inverses and then use what is known as {\sl short lex order}.

\subsection{The rewriting system $\tau$}

We begin with a more general situation and assume $\Gamma_0$ has presentation given by generators $w_1,...,w_q$ with relations $R_u(w_1,...,w_v)$ for some integers $q$, $u$  and $v$ where the $R_u$ are words in the generators.


We remind the reader that a Reidemeister-Schreier rewriting process $\tau$ writes a word in the generators of $\Gamma_0$ that lies in $\Gamma$ in terms of the generators $S_{K,v}$ for $\Gamma$ where $K$ runs over a complete set of Schreier  minimal coset representatives. 

We let ${\overline{X}}$ denote the coset representative of $X$ in $\Gamma_0$. That is, the element $K$ where $\phi(K) = \phi(X)$. The element $S_{K,v} \in \Gamma$ is the element $Kv{\overline{Kv}}^{-1}$.

The rewriting process is defined as follows:

For each integer $i =1, ..., u$ let $d_i$ be one of  $\{w_1,..., w_q\}$.

Let $X= d_1^{\epsilon _i} \dots d_u^{\epsilon_u}$, where each $\epsilon_i$ is either $+1$ or $-1$.
Then $$\tau(X) = \Pi_{i=1}^u S_{V_i,d_i}^{\epsilon_i},$$
 where $V_i$ depends upon $\epsilon_i$. Namely, if $\epsilon_i = +1$, then $V_i = {\overline{d_1^{\epsilon_1} \dots d_{i-1}^{\epsilon_{i-1}}} }$ and if
$\epsilon_i = -1$, then $V_i = {\overline{d_1^{\epsilon_1} \dots d_{i-1}^{\epsilon_{i-1}} d_i^{-1}}}$.

We next apply the rewriting system to the case for $\Gamma$ and $\Gamma_0$ under consideration in this paper.

\section{Application of the Reidemeister-Schreier Theorem} \label{sec:SR}

Apply  the Reidemeister-Schreier Theorem to our situation to obtain:

\begin{thm}
\label{thm:SR} {\rm(Presentation for $\Gamma$ with Schreier generators)}

\noindent Let $\Gamma_0$ have generators
 $$ a_1,...,a_{g_0}, b_1,...,b_{g_0}, x_1, ...x_r $$
and relations
 $$ R=  \Pi_{i=1}^{g_0} [a_i,b_i]  \cdot x_1 \cdots x_r =1,    \;\;\; x_i^{m_i}=1$$

\noindent and let $\Gamma$ be the subgroup of $\Gamma_0$ with $G$ as above isomorphic to  $\Gamma_0 / \Gamma$.
\vskip .03in

\noindent Then $\Gamma$ can be presented with generators 

$$  S_{K,a_i}, \;\; S_{K,b_i}, i =1,...,g_0,   \mbox{ and } \;\; S_{K,x_j},  j=1,...,r$$
and relations

$$\tau(KRK^{-1}) = 1,\;\; \tau(Kx_j^{m_j}K^{-1})=1,\;\; j = 1,...,r, \;\; S_{M,v}=1$$
where $M$ is  Schreier representative and $v$ is  any  generator such that $Mv={\overline{Mv}}$

\end{thm}

\begin{cor} \label{cor:numbers}

If $\tau$ is a Reidemeister-Schreier rewriting process,

then $\Gamma$ can be presented as a group with $2ng_0 + nr$  generators and

$n - (n-1) + \Sigma_{j=1}^r {\frac{n}{m_i}}$ non-conjugate relations.
\end{cor}

\begin{proof}
We note that the relations $\tau(KRK^{-1})$ are $n$ non-conjugate relations, and the $\tau(Kx_j^{m_j}K^{-1})$ add another non-conjugate $\Sigma_{i=1}^r {\frac{n}{m_i}}$ because each of the latter relations contains $m_i$ elements.  There must be ${\frac{n}{m_i}}$ distinct non-conjugate relations. However, there are $n-1$ minimal Shreier coset representatives other than the identity and these give the relations of type $S_{M,v}=1$ so that $n-1$ must be subtracted from the sum.
More precisely, if $N= c_1c_2 \cdots c_{q-1}c_q$ is a chosen coset representative, then its initial segment $M = c_1 \cdots c_{q-1}$ is also a chosen coset representative and thus $S_{M,c_q} =1$.

We note that for fixed $j$ and $m_j$, the $\tau(Kx_j^{m_j}K^{-1})$ and $\tau({\overline{Kx_j^n}}x_j^{m_j}{\overline{Kx_j^n}}^{-1})$ yield the same set of relations.

\end{proof}

\subsection{Action of $G$ on the kernel} \label{sec:kernelLIFT}

\vskip .03in
Next we translate from the language of algebra to curve lifting.
\vskip .03in
We note that the action of $G$ or $\Gamma_0$ on $\Gamma$ is given by conjugation and if we set $\phi(K) = g_K$, then
for any word $W$ in $\Gamma$, we define the action of $g_K$ and $g_{\phi(K)}$ by the equation $g_K(W) = g_{\phi(K)}(W) = \tau(KWK^{-1})$.

\begin{rem} For any generator $V$ of $\Gamma_0$, $S_{1,V}$
corresponds to the lift of
$V{\overline{V}}^{-1}$ to initial point $p \in S$  where $\pi(p) = p_0 \in S_0$. Let $W$ be a curve in $\Gamma_0$ and ${\tilde{p}}$ its end point when lifted to $S$ with initial point $p$.  Then $S_{W,V}$ is the lift of $V{\overline{V}}^{-1}$ to a curve with initial point, call it  ${\tilde{p}}$ and $g_U(S_{W,V})$ is the lift of $S_{W,V}$ to a curve whose initial point is the end point of $U$ when lifted to $p$.
\end{rem}

\begin{cor}
\label{lem:action} Let $X$ be a coset representative from a Schreier system of right coset representatives and $S_{U,V}$ an element of the kernel.
Assume that $g_X$ represents the action of $\phi(X)$ on the kernel $\Gamma$. Then
$g_X(S_{U,V}) = S_{{\overline{XU}},V}$.
\end{cor}

\begin{rem} We note that this can also be shown using the $\tau$ notation and properties found on page 89 of \cite{MKS} by calculating $\tau(XUV{\overline{UV}}^{-1}X^{-1})$ and reducing.

\end{rem}

\subsection{Adapted homology bases} \label{sec:homology}
\begin{defn} \label{def:AHB} {\rm {Adapted   Homology Basis}}

Let $\gamma$ be an arbitrary curve in $\mathcal{B}$, an integral homology  basis,  $H_1(S)$,
 for $S$,  and let $G$ be  a group of conformal automorphisms of $S$ of order $n$ or equivalently a finite subgroup of the mapping-class group. Then a set of $2g$ generators  ${\mathcal{S_B}}$
 for the integral homology  $\mathcal{B}$,  is {\sl adapted to} $G$ if for each $\gamma \in {\mathcal{S_B}}$ one of the following occurs:

\begin{enumerate}
 \item \label{item:prop1}  $\gamma$ and $g(\gamma)$ are in ${\mathcal{S_B}}$ for all $g \in G$ and $g(\gamma) \ne \gamma$.

 \item $\gamma$ and $h^j(\gamma)$ are  in ${\mathcal{S_B}}$  for all $j=0, 1,...,m_{i-2}$ where $ h \in G$ is of order $m_i$, and
           \label{item:prop2}
      $$h^{m_{i-1}}({\gamma})  =  - (\gamma + h(\gamma) + \cdots  + h^{m_{i-2}}(\gamma)).$$
 Further
       for each coset representative, $g_h$,  for $G$ modulo $\langle h \rangle$, we have that
         $g_h(\gamma)$ and $(g_h \circ h^j)(\gamma)$ are in the set  for all $j=0, 1,...,m_{i-2}$ and
         $$(g_h \circ h^{m_{i-1}})({\gamma})  =  - (g_h(\gamma) + (g_h \circ h)(\gamma) + \cdots (g_h \circ h^{m^{i-2}})(\gamma)$$

     \item $\gamma = h^r(\gamma_0)$ where $\gamma_0$ is one of the curves in item \ref{item:prop2} above.

         \item \label{item:prop3} $g(\gamma) = \gamma$ for all $g \in G_0$, $G_0$ a subgroup of $G$ of order $m$.

          All of the other $n/m$ images of $\gamma$ under $G$ are fixed appropriately by conjugate elements or  by elements representing the cosets of $G/{G_0}$ and are also in ${\mathcal{S_B}}$.
\end{enumerate}

\end{defn}
\section{The symmetric group ${\mathcal{S}}_3$ and its multiplication table} \label{sec:S_3}

We begin with the $6$ generators:

        $id$,   $A =(1,2)$, $B = (2,3)$, $C=(1,3)$, $D=(1,2,3)$,  $E= (1,3,2)$

Multiplication gives
 \begin{itemize}
 \item $A$: $\;$   $A^2=id, AB = D, AC= E, AD = B, AE = C$
 \item $B$:  $\;$ $BA= E, B^2=id, BC = D, BD= C, BE = A$
 \item $C$: $\;$ $CA =D, CB=E, C^2 = id, CD= A, CE = B$
 \item $D$: $\;$ $DA =C, DB=A, DC= B, D^2=E, DE = id$
 \item $E$: $\;$ $EA=B, EB = C, EC = A, ED=id, E^2=D$
\end{itemize}

 \subsection{Coset Representatives:}
We assume first that $g_0=0$ and  that $\Gamma_0$ has presentation

$$\langle x_1, x_2,x_3, x_4,x_5,x_6,x_7, x_8 \;  | \; x_1x_2x_3x_4x_5x_6x_7 x_8=1; x_i^2=1, i=1,... 6;  x_7^3 = x_8^3=1\rangle .$$

We assume that $\phi(x_1) = \phi(x_2) = A$, $\phi(x_3) = \phi(x_4) = B$, $\phi(x_5) = \phi(x_6) =C, $ $\phi(x_7) = D$, $\phi(x_8) = E$.
We are using $A, B, C, D$ and $E$ to represent the right coset representatives.

       We pick coset representatives $x_1,x_3,x_5, x_7, x_8$ whose images are respectively  $A$,$B$, $C$, $D$, $E$

\subsection{Generators and Relations}

The generators  of $\Gamma$ are
\vskip .02in
$S_{1,x_1}, S_{1,x_2},S_{1,x_3},S_{1,x_4},S_{1,x_5},S_{1,x_6},S_{1,x_7},S_{1,x_8}$

$S_{A,x_1}, S_{A,x_2},S_{A,x_3},S_{A,x_4},S_{A,x_5},S_{A,x_6},S_{A,x_7},S_{A,x_8}$

$S_{B,x_1}, S_{B,x_2},S_{B,x_3},S_{B,x_4},S_{B,x_5},S_{B,x_6},S_{B,x_7},S_{B,x_8}$

$S_{C,x_1}, S_{C,x_2},S_{C,x_3},S_{C,x_4},S_{C,x_5},S_{C,x_6},S_{C,x_7},S_{C,x_8}$

$S_{D,x_1}, S_{D,x_2},S_{D,x_3},S_{D,x_4},S_{D,x_5},S_{D,x_6},S_{D,x_7},S_{D,x_8}$

$S_{E,x_1}, S_{E,x_2},S_{E,x_3},S_{E,x_4},S_{E,x_5},S_{E,x_6},S_{E,x_7},S_{E,x_8}$
\vskip .2in

The relations are
\vskip .1in
$\tau(R) =1$

$\tau(ARA^{-1})=1$

 $\tau(BRB^{-1})=1$

  $\tau(CRC^{-1})=1$

$\tau(DRD^{-1})=1$

$\tau(ERE^{-1})=1$

$\tau(x_i^2) =1$ $i =1,...,6$

$\tau(x_7^3)=1$

$\tau(x_8^3)=1$
\subsection{Results of Computations} \label{sec:resultsofcomp}

We compute $$\tau(x_7^3) = S_{1,x_7}S_{D,x_7}S_{E,x_7}.$$
$$\tau(x_1x_7^3x_1^{-1}) = S_{A,x_7}S_{B,x_7}S_{C,x_7}.$$

Since  $$S_{1,x_7} = 1, S_{D,x_7}^{-1} =S_{E,x_7}.$$
and
 $$S_{A,x_7}S_{B,x_7}= S_{C,x_7}^{-1}.$$

\subsection{Genus count}
$g_0=0$ $n = 6$ There are $6$ elliptics of order $2$ and $2$ of order $3$. We have

$$2g-2= 6(-2) + 6 ( 6(1-{\frac{1}{2}} = 6(2 (1- {\frac{1}{3}} = -12 + 2(4) +18 = -12 + 26$$ so that

$$2g = -10 +26 =16.$$ Thus

$$g =8. $$

\section{Detailed Calculations for ${\mathcal{S}}_3$}

Using $\tau(R) =1 $ and $\tau(x_Rx_1^{-1}) =1$  and combining with the $\tau(x_8^3)=1$ and $\tau(x_1x_8^3x_1^{-1})=1$ and then  using the solutions for $S_{*,x_8}=1$ where $* \in \{ 1, A,B,C,D,  E \}$, we have two relations.
\vskip .2in
The first is:

$$ S_{A,x_1}S_{1,x_2}S_{A,x_3}S_{D,x_4}S_{A,x_5}S_{E,x_6}S_{A,x_7} \times $$  

$$S_{B,x_1}S_{E,x_2}S_{B,x_3}S_{1,x_4}S_{B,x_5}S_{D,x_6}S_{B,x_7} \times $$ 

$$S_{C,x_1}S_{D,x_2}S_{C,x_3}S_{E,x_4}S_{C,x_5}S_{1,x_6}S_{C,x_7} =1$$


\vskip .2in

 Solving for $S_{D,x_8}^{-1}$,  $S_{E,x_8}^{-1}$,  and $S_{1,x_8}^{-1}$ gives us the 
 second combined three rows:
\vskip .2in
$$S_{1,x_1}S_{A,x_2}S_{1,x_3}S_{B,x_4}S_{1,x_5}S_{C,x_6}S_{1,x_7} \times $$
$$S_{Dx_1}S_{C,x_2}S_{D,x_3}S_{A,x_4}S_{D,x_5}S_{B,x_6}S_{D,x_7} \times $$
$$S_{E,x_1}S_{B,x_2}S_{E,x_3}S_{C,x_4}S_{E,x_5}S_{A,x_6}S_{E,x_7} =1  $$

\vskip .2in
Using $S_{A,x_1} =1=S_{1,x_1}$
$S_{1,x_3} = 1 = S_{C,x_3}$
$S_{D,x_7}  S_{E,x_7}^{-1}$
$S_{1,x_7} =1 $ and
$S_{1,x_5} = 1 =S_{C,x_5}$ we obtain

$$S_{1,x_2}S_{A,x_3}S_{D,x_4}S_{A,x_5}S_{E,x_6}S_{A,x_7} \times $$   

$$S_{C,x_1}S_{D,x_2}S_{C,x_3}S_{E,x_4}S_{1,x_6}S_{C,x_7} =1$$

  $$S_{A,x_2}S_{B,x_4}S_{C,x_6}S_{1,x_7} \times$$
$$S_{D, x_1}S_{C,x_2}S_{D,x_3}S_{A,x_4}S_{D,x_5}S_{B,x_6}S_{D,x_7} \times $$
$$S_{E,x_1}S_{B,x_2}S_{E,x_3}S_{C,x_4}S_{E,x_5}S_{A,x_6}S_{E,x_7} =1  $$

\vskip .2in
We combine
 into one relation by identifying $S_{1,x_2}$ and its inverse $S_{A,x_2}$
\vskip .2in

We obtain the relation $\mathcal{R}$:

$$(S_{B,x_4}S_{C,x_6}S_{1,x_7} \times $$

$$S_{D, x_1}S_{C,x_2}S_{D,x_3}S_{A,x_4}S_{D,x_5}S_{B,x_6}S_{D,x_7} \times $$

$$S_{E,x_1}S_{A,x_7}S_{E,x_3}S_{C,x_4}S_{E,x_5}S_{A,x_6}S_{E,x_7})^{-1}$$

$$\times S_{A,x_3}S_{D,x_4}S_{A,x_5}S_{E,x_6}S_{A,x_7} \times$$   
$$S_{B,x_1}S_{E,x_2}S_{B,x_3}S_{1,x_4}S_{B,x_5}S_{D,x_6}S_{B,x_7} \times $$ 
$$S_{C,x_1}S_{D,x_2}S_{E,x_4}S_{1,x_6}S_{C,x_7} =1$$

\vskip .2in
Use the fact that $S_{1,x_8}=1$, to solve for $S_{E,x_5}$.
\vskip .1in
That is, use
$S_{E,x_1}S_{B,x_2}S_{E,x_3}S_{C,x_4}S_{E,x_5}S_{A,x_6}S_{E,x_7} = S_{E,x_5}^{-1}$ and then substitute into the equations above to obtain  the $16$ generators that together with their inverses appear in the one relation:

$$S_{B,x_4}, S_{C,x_6}, S_{D, x_1}, S_{C,x_2}, S_{D,x_3}, S_{C,x_7},S_{A,x_4}, S_{D,x_5}$$

$$ S_{B,x_6}, S_{D,x_7}, S_{E,x_1},S_{A,x_7}, S_{E,x_3}S_{C,x_4},  S_{A,x_6}, S_{E,x_7}$$

Denote the new relation from which $S_{E,x_5}$ has been eliminated by $${\mathcal{RR}}.$$

We note that $S_{1,x_7}=1$ and $S_{E,x_7} = S_{D,x_7}^{-1}$

We conclude:

\begin{cor} \label{cor:gen-rel}
$\Gamma$ is generated by the $16$ elements

$$S_{B,x_4}, S_{C,x_6}, S_{A, x_7}, S_{C,x_2}, S_{D,x_1}, S_{D,x_3}, S_{C,x_7}, S_{A,x_4}, S_{D,x_5}$$

$$ S_{B,x_6}, S_{D,x_7}, S_{E,x_1},S_{E,x_3}S_{C,x_4},  S_{A,x_6}S_{E,x_7}$$ and has a single defining relation, the
relation $\mathcal{RR}$.
\end{cor}

Since these $16$ elements and their inverses occur in the single defining relation, using the algorithm from \cite{Springer}, these can be replaced by $16$ generators with a single relation that is a product of commutators. However, these new generators will not exhibit the action $G$ on the kernel as nicely (see Remark {\ref{rem:LINK}).

\subsection{Action of $G$ on the kernel and matrices}

\begin{prop} \label{prop:matrices}
The matrix  of the induced actions of the elements of $G$ on homology, 
 is a  $16 \times 16$ matrices with entries $0, 1,  \mbox{and} -1.$
 Only one row has more than one non-zero entry. All other rows have one non-zero entry.
\end{prop}

\begin{proof}

We can compute the images under $g_A$ of the generators using $g_A(S_{U,V})=S_{{\overline{AU}},V}$. Note that $g_A(S_{D,x_7}) = S_{B,x_7} =S_{A,x_7}^{-1}S_{C,x_7}^{-1}$, but all other images are of length one in the generators.
That is, for example,  $g_A(S_{B,x_4}) = S_{D,x_4}= S_{A,x_4}^{-1}$

 \noindent We note that inverses can be verified by using $\tau$, $\phi$, and the coset representatives.

 \noindent For example to see that $\tau(S_{D,x_4}S_{A,x_4}) = 1$, consider $Dx_4{\overline{Dx_4}}^{-1}Ax_4{\overline{Ax_4}}^{-1}$.

 \noindent Compute that  $Dx_4{\overline{Dx_4}}^{-1}Ax_4 {\overline{(AB)}}^{-1}= Dx_4A^{-1}Ax_4D^{-1}= Dx_4x_4D^{-1}=1$ since $x_4$ is of order $2$.

We replace the $16$ generators for the fundamental group by their images in the first homology group.

Note that the each row in the corresponding matrix for the action on the $16$ generators for the first homology has either $1$ or $2$ non zero entries. Each non-zero entry is with $1$ or $-1$.
There is only one row with more  than one non-zero entry.
Similar computations are easily carried out for $g_B,g_C,g_E$ and $g_D$.

\end{proof}

\begin{rem}\label{rem:LINK}
A pair of generators $u$ and $v$ that appear in a relation in a group are said to be linked if they appear in the order
   $...u....v...u^{-1}...v^{-1}...$ (see p 120 of \cite{Springer}) A set of $2g$ generators for the fundamental group of a genus $g$ compact surface in which every generator and its inverse occurs must be {\it linked}. If  the fundamental is given by $2g$ generators with a single defining relation in which every generator is linked with one other generator, it can be presented as a set of $2g$ generators whose defining relation is  a product of commutators.

\end{rem}

\begin{cor} \label{cor:adaptedgens}
The images in the first homology group of the $16$ generators listed form a basis for the first homology group, but do not form  an adapted homology basis for  $\Gamma$. 
\end{cor}

\begin{proof}
We consider their images in the first homology group and note that these are distinct and the dimension is correct, We  note that the images under $G$ of all of lifts of $S{C,x_2}$ do not appear. Thus this is not an adapted basis.
Calculations show that $16$ generators are linked.
\end{proof}
\newpage

\subsection{The action of $\mathcal{S}_3$ when  $g_0 \ne 0.$}

\noindent  We consider the action of $\mathcal{S}_3$ in the case when $g_0 \ne 0.$

\noindent Let $\Omega_0$ be a group with generators
$$ a_1,....a_{g_0},b_1,...,b_{g_0}, y_1, y_2, y_3, y_4, y_5, y_6, y_7, y_8$$ and relations
$$(\Pi_{i=1}^{g_0}[a_i,b_i])y_1y_2y_3y_4y_5y_6y_7 y_8=1; y_i^2=1, i=1,... 6;  y_7^3 = y_8^3=1. $$

\noindent Let $G = {\mathcal{S}}_3$ be as in section \ref{sec:S_3} and let ${\hat{\phi}}: \Omega_0 \rightarrow G$ with  ${\hat{\phi}}(y_1) = {\hat{\phi}}(y_2) = A$, ${\hat{\phi}}(y_3) = {\hat{\phi}}(y_4) = B$, ${\hat{\phi}}(y_5) = \hat{\phi}(y_6) =C $ $\hat{\phi}\phi(y_7) = D$, $\hat{\phi}(y_8) = E$.
That is, assume that ${\hat{\phi}}$ and $\phi$ agree on the corresponding elliptic generators of $\Omega_0$ and $\Gamma_0$. The images of the hyperbolic generators under ${\hat{\phi}}$ are easy to control since they lie in a cyclic group of order three.

\begin{lem} \label{cor:hatgen-rel}
${\hat{\Gamma}}$ is generated by the $16$ elements
$$S_{B,y_4}, S_{C,y_6}, S_{A, y_7}, S_{C,y_2}, S_{D,y_1}, S_{D,y_3}, S_{C,y_7},
S_{A,y_4}, S_{D,y_5}$$

$$ S_{B,y_6}, S_{D,y_7}, S_{E,y_1},S_{A,y_7}, S_{E,y_3}S_{C,y_4},  S_{A,y_6}S_{E,y_7}$$ along with
 the $2g_0n$ elements in the set  $\{ (S_{K,a_i}, S_{K,b_i}) i = 1,...,g_0 \}$, the lifts of the hyperbolics. It has a single defining relation,
the relation ${\widehat{\mathcal{RR}}}$.
\end{lem}

\begin{proof}
Use theorem \ref{thm:SR} and repeat the calculations for $g_0$ carrying through the $\tau(K(\Pi_{i=1}^{g_0}[a_i,b_i])K^{-1})$. The latter are not involved in any of the eliminations. The calculation replaces $\mathcal{RR}$ and obtains a new single relation ${\widehat{\mathcal{RR}}}$ at the end.
\end{proof}

\begin{thm} \label{thm:extmore}  Let ${\hat{\phi}}$ and $\phi$ be surface kernel maps from $\Omega_0$ and $\Gamma_0$ respectively into $\mathcal{S}_3$. Assume that ${\hat{\phi}}$ and $\phi$ agree on the images of corresponding elliptic elements so that ${\hat{\phi}}(y_j) = \phi(x_j)$ for all $j = 1,...,r.$. Let $\hat{\Gamma}$ and $\Gamma$ be their respective kernels.
  Then generators for the commutator quotient of  ${\hat{\Gamma}}$ can be given by the images of the lifts of hyperbolic elements together with the images of the $16$ elements given in corollary \ref{cor:hatgen-rel}. The properties of the
  the $16$ elements in $\hat{\Gamma}$ and $\Gamma$ with respect to inverses and under the action of $G= {\mathcal{S}}_3$ are the same with $x_i$ replaced by $y_i$ for $i=1,...,8$.

\end{thm}
\begin{proof} If $v$ is any hyperbolic generator,  then $\{ S_{K,v} \}$ as $K$ varies over the  $6$  coset representatives consists of  the lift of $S_{1,v}$ and all of its images under $G$. The last statement is an easy translation from calculation for $\Gamma$. That is, $\tau(S_{D,x_4}S_{A,x_4})=1$ becomes $\tau(S_{D,y_4}S_{A,y_4})=1$.
\end{proof}

Since  $g = 2ng_0 + 16$, it follows immediately that:

\begin{thm}
The matrix of the induced action on homology is a $g \times g$ matrix that breaks up into four blocks: One block is $2ng_0 \times 2ng_0$ and is a permutation matrix. A second block is a $16 \times 16$ block and is the same as the matrix obtained in the case $g_0 =0$. The entries in the other two blocks are  all $0$.

\end{thm}

\begin{rem}
Since the order of $\mathcal{S}_3$ is $6$, we can replace $2ng_0 + 16$ by $12g_0 + 16$ in all of the results above.
\end{rem}

\section{Questions} \label{sec:questions}

\noindent We refer to the automorphisms in Harvey's list as {\it Harvey operations}. We ask about {\it normal forms} for generating vectors. In \cite{BW} the authors use Harvey operations in the abelian case to find normal forms for generating vectors and in \cite{BW1} an example of using non-abelian Harvey operations for an action on ${\mathcal{S}}_3$ in genus $3$ is given. These two papers use Birman's work \cite{Birman}.
\vskip .2in

\noindent We pose some questions below.

\vskip .2in
\begin{enumerate}
\item {\sl Normal Forms}

\subitem For non-abelian $G$,  can we define a {\sl normal form} or {\sl simplest form} for a surface kernel map that corresponds to an equivalence class?

\subitem If so, can normal forms be easily enumerated and can the non-abelian version of Harvey's lift (or extensions) be efficiently used to compute the normal form?

\subitem What restrictions on $G$ (or on $\phi$ as in Theorem \ref{thm:furtherext}) allows for a more tractable equivalence problem?

    \subitem Can a {\it normal form} for a generating vector be defined that corresponds to an  equivalence class? Harvey did such for cyclic groups. Can this be done for any nonabelian group $G$?

     \subitem Can the Harvey operations be efficiently used to compute equivalence classes?

     \item Can we make restrictions on surface kernel maps (as in  Theorem \ref{thm:furtherext}) to make the Harvey operations simpler. Do the restrictions have a geometric interpretation beyond the initial interpretation with respect to curve lifting?

     \item  Can one apply the S-R theory using a set of coset representatives that are not a minimal Schreier set and obtain useful results?

\item  Are there are other useful extensions of Harvey's  surface kernel maps? If so, what are some?

    \end{enumerate}

\section{Acknowledgement} The author thanks the referee for helpful comments and for suggested extensions and improvements of some results.   \bibliographystyle{amsplain}

\end{document}